\documentclass{article}
\usepackage{amsmath}
\usepackage{amssymb}
\usepackage{amsfonts}
\usepackage{eucal}
\usepackage{url}
\usepackage{bm}
\usepackage[dvipdfmx]{hyperref}
\usepackage{mathrsfs}
\usepackage[dvipdfmx]{graphicx}
\usepackage{amscd}
\usepackage{amsthm}
\usepackage[all]{xy}
\usepackage{color}

\newcommand{\Hom}{\mathop{\mathrm{Hom}}\nolimits}

\newcommand{\CAlg}{\mathop{\mathrm{CAlg}}\nolimits}
\newcommand{\Ker}{\mathop{\mathrm{Ker}}\nolimits}

\newcommand{\cmod}{\textrm{-}\mathrm{mod}}
\newcommand{\Cone}{\mathrm{Cone}}
\newcommand{\Cyl}{\mathrm{Cyl}}
\newcommand{\Path}{\mathrm{Path}}

\newcommand{\cA}{\mathcal{A}}
\newcommand{\cB}{\mathcal{B}}
\newcommand{\cC}{\mathcal{C}}
\newcommand{\cD}{\mathcal{D}}
\newcommand{\cF}{\mathcal{F}}

\newcommand{\cP}{\mathcal{P}}

\newcommand{\cV}{\mathcal{V}}

\newcommand{\fg}{\mathfrak{g}}

\newcommand{\fk}{\mathfrak{k}}

\newcommand{\fS}{\mathfrak{S}}

\newcommand{\bL}{\mathbb{L}}

\newcommand{\bR}{\mathbb{R}}

\theoremstyle{plain}
\newtheorem{thm}{Theorem}[subsection]
\newtheorem{cor}[thm]{Corollary}
\newtheorem{lem}[thm]{Lemma}
\newtheorem{prop}[thm]{Proposition}
\theoremstyle{definition}

\newtheorem{cons}[thm]{Construction}

\newtheorem{ex}[thm]{Example}
\newtheorem{note}[thm]{Notation}

\newtheorem{rem}[thm]{Remark}
\begin{document}
\title{Dg analogues of the Zuckerman functors and the dual Zuckerman functors II}
\author{Takuma Hayashi}
\date{}
\maketitle
\begin{abstract}
We construct derived functors of dg analogues of the Zuckerman functors and the dual Zuckerman functors in view of the theory of model categories.
\end{abstract}
\section{Introduction}
This is the second part of the series of papers on dg analogues of the Zuckerman functors and the dual Zuckerman functors, following \cite{H1}. We have three subjects in this series: The first subject is to introduce and study fundamental properties of their categories of dg analogues of weak (Harish-Chandra) pairs, pairs, weak triples and triples over an arbitrary commutative ring, and modules over them, which was partly achieved in the previous part \cite{H1}; The second subject is to define dg analogues and their weak analogues $I^{\cB,L}_{\cA,K,w}$, $I^{\cB,L}_{\cA,K}$, $P^{\cB,L}_{\cA,K,w}$, and $P^{\cB,L}_{\cA,K}$ of induction functors, production functors, the Zuckerman functors, and the dual Zuckerman functors, which was done in \cite{H1}; The third subject is to define their derived functors.

The main theme of this paper is to establish the third subject. The first subject will be also discussed as a supplement at the beginning. This project was partly done by \cite{P1}, \cite{Ja1}, \cite{Ja2}, and \cite{Ja3} by means of homological algebra. Our approach here is based on the theory of model categories introduced by Daniel Quillen in \cite{Q1}. This will suit to the context of unbounded derived functors and (stable) $\infty$-categories (\cite{L2}). The main goal is to construct combinatorial model structures on the categories of modules over pairs and weak pairs, which are appropriate to the aim.
\subsection{Reviews on \cite{H1}}
In this paper we follow the same conventions as those of the former paper \cite{H1}. Here we shall begin with a brief summary on loc\ cit. Let $k$ be a (small) commutative ring, and let $k\cmod$ denote the category of cochain complexes of $k$-modules. For a flat affine group scheme $K$ over $k$, let $K\cmod$ denote the category of cochain complexes of $K$-modules. A weak pair is a pair $(\cA,K)$ of a flat affine group scheme $K$ and a monoid object $\cA$ of $K\cmod$. If one wants to emphasize the corresponding action map $\phi$ of $K$ on $\cA$, this weak pair is denoted by $(\cA,K,\phi)$, and it is called a weak tuple. For a weak pair $(\cA,K)$, we define the category $(\cA,K)\cmod_w$ of (left) weak $(\cA,K)$-modules as that of left modules over $\cA$ in the sense of the theory of monoidal categories. For a weak $(\cA,K)$-module $V$, the corresponding differential and actions of $\cA$ and $K$ will be dented by $d=d_V$, $\pi=\pi_V$, and $\nu=\nu_V$ respectively. Morphisms of weak pairs are also defined by a general context of monoidal categories. Namely, for weak pairs $(\cA,K)$ and $(\cB,L)$ a weak map from $(\cA,K)$ to $(\cB,L)$ is a pair $f=(f_a,f_k)$ consisting of a group scheme homomorphism $f_k:K\to L$ and a $K$-equivariant dg algebra homomorphism $f_a:\cA\to\cB$ via $f_k$. For a weak map $(\cA,K)\to(\cB,L)$, we have the ``forgetful functor''
\[\cF^{\cA,K}_{\cB,L,w}:(\cB,L)\cmod_w\to(\cA,K)\cmod_w.\]
\begin{lem}[\cite{H1} Proposition 2.3.6]\label{1.1.1}
For a weak map
\[(f_a,f_k):(\cA,K)\to(\cB,K)\]
with $f_k$ the identity map, the functor $\cF^{\cA,K}_{\cB,L,w}$ admits a left adjoint functor $P^{\cB,K}_{\cA,K,w}$.
\end{lem}
We next consider their ``non-weak'' analogues. In this case our affine group schemes are supposed to satisfy \cite{H1} Condition 2.2.2. This is a condition to define the adjoint representation Ad. A pair consists of a weak pair $(\cA,K)$ with $K$ satisfying \cite{H1} Condition 2.2.2 and a $K$-equivariant dg Lie algebra homomorphism $\psi:\fk\to\cA$ such that the following equality holds for any $\xi\in\fk$:
\[d\phi(\xi)=\left[\psi(\xi),-\right]:\cA\to\cA.\] 
Here $\phi$ denotes the corresponding action of $K$ on $\cA$. See \cite{H1} 2.1 for the rest of notations. If we want to emphasize the maps $\phi$ and $\psi$, this is denoted by $(\cA,K,\phi,\psi)$, and it is called a tuple. For a tuple $(\cA,K,\phi,\psi)$, a weak $(\cA,K)$-module is called an $(\cA,K)$-module if it satisfies the following equality:
\[d\phi(\xi)=\left[\psi(\xi),-\right]:\cA\to\cA.\]
The full subcategory of $(\cA,K)\cmod_w$ spanned by $(\cA,K)$-modules is denoted by $(\cA,K)\cmod$ in this paper.
\begin{lem}[\cite{H1} Lemma 2.3.2]\label{1.1.2}
For a pair $(\cA,K)$, the category $(\cA,K)\cmod$ is both a localization and a colocalization of $(\cA,K)\cmod_w$.
\end{lem}
We next define maps of pairs. Let $(\cA,K,\phi_\cA,\psi_\cA)$
and $(\cB,L,\phi_\cB,\psi_\cB)$
be tuples. Then a map $(\cA,K)\to(\cB,L)$ of pairs is a weak map $f=(f_a,f_k)$ respecting $\psi$, i.e., satisfying the equality
\[f_a\circ\psi_\cA=\psi_\cB\circ df_k:\fk\to\cB.\]
By easy computations, we showed the following results:
\begin{lem}[\cite{H1} Proposition 2.3.1, Proposition 2.3.6]\label{1.1.3}
\begin{enumerate}
\renewcommand{\labelenumi}{(\arabic{enumi})}
\item Let
\[(\cA,K)\to(\cB,L)\]
be a map of pairs. Then the restriction of $\cF^{\cA,K}_{\cB,L,w}$ gives rise to the ``forgetful'' functor
\[\cF^{\cA,K}_{\cB,L}:(\cB,L)\cmod\to(\cA,K)\cmod.\]
\item For a map $(f_a,f_k):(\cA,K)\to(\cB,K)$ of pairs with $f_k$ the identity map, the restriction of $P^{\cB,K}_{\cA,K,w}$ gives rise to a functor
\[P_{\cA,K}^{\cB,K}:(\cA,K)\cmod\to(\cB,K)\cmod.\]
Moreover, $P_{\cA,K}^{\cB,K}$ is a left adjoint functor of $\cF^{\cA,K}_{\cB,K}$.
\end{enumerate}
\end{lem}
We also constructed functors $P^{\cB,L}_{\cA,K,w}:(\cA,K)\cmod_w\to(\cB,L)\cmod_w$ and $P^{\cB,L}_{\cA,K}:(\cA,K)\cmod\to(\cB,L)\cmod$ in the cases where the base ring $k$ is the complex number field and the groups $K,L$ are reductive in \cite{H1}.
\begin{prop}[\cite{H1} 2.5]\label{1.1.4}
The functors $P^{\cB,L}_{\cA,K,w}$ and $P^{\cB,L}_{\cA,K}$ admit exact right adjoint functors.
\end{prop}
\subsection{Backgrounds}
Before stating the main results let us begin by introducing their prototypes in classical homological algebra. Let $f:A\to B$ be a homomorphism of rings, and let $A\cmod$ (resp.\ $B\cmod$) denote the category of cochain complexes of $A$-modules (resp.\ $B$-modules). It is well-known that the restriction functor $f^\ast:B\cmod\to A\cmod$ admits both the left adjoint functor $B\otimes_A-$ and the right adjoint functor $\Hom_A(B,-)$. An approach to define its (unbounded) derived functors is to construct the so-called injective and projective model structures.
\begin{thm}[\cite{Hov1}]\label{1.2.1}
The category $A\cmod$ admits a proper combinatorial model structures satisfying the following conditions:
\begin{enumerate}
\item[(C)]A map is a cofibration if and only if it is a monomorphism.
\item[(W)]A map is a weak equivalence if and only if it is a quasi-isomorphism.
\end{enumerate}
This is called the injective model structure. Moreover, the functor $f^\ast$ is left Quillen with respect to this model structure. In particular, the unbounded derived functor $\bR\Hom_A(B,-)$ exists.
\end{thm}
More generally, such a result is valid for the category $C(\cA)$ of cochain complexes of objects of a Grothendieck abelian category $\cA$.
\begin{thm}[\cite{Be}, \cite{Hov2}, \cite{L2}]\label{1.2.2}
The category $C(\cA)$ admits a proper combinatorial model structures satisfying the following conditions:
\begin{enumerate}
\item[(C)]A map is a cofibration if and only if it is a monomorphism.
\item[(W)]A map is a weak equivalence if and only if it is a quasi-isomorphism.
\end{enumerate}
\end{thm}
As a variant, it is also known that for a differential graded ring $A$ the category of $A\cmod$ admits a similar model structure (see \cite{A} for example).

We next get into the projective model structure.
\begin{thm}[\cite{Hov1}]\label{1.2.3}
Let $A$ be a ring. Then the category $A\cmod$ admits a proper combinatorial model structures satisfying the following conditions:
\begin{enumerate}
\item[(F)]A map is a fibration if and only if it is an epimorphism.
\item[(W)]A map is a weak equivalence if and only if it is a quasi-isomorphism.
\end{enumerate}
This is called the projective model structure. Moreover, if $A$ is commutative it is a symmetric monoidal model category satisfying the monoid axiom in the sense of \cite{SS} for $\otimes_A$.
\end{thm}
\begin{cor}\label{1.2.4}
If we are given a homomorphism $f:A\to B$ of rings the functor $f^\ast$ is right Quillen with respect to this model structure. In particular, the unbounded derived functor $B\otimes^\bL_A-$ exists.
\end{cor}
To generalize it to dg rings we can use the transfer method:
\begin{thm}[transfer of a model structure, \cite{GS} Theorem 3.6]\label{1.2.5}
Let $\cC$ be a combinatorial model category, $\cD$ a locally presentable category, and let $(F,U):\cC\to\cD$ be an adjunction. Suppose that if a map $f$ of $\cD$ has the left lifting property with respect to any map $p$ of $\cD$ with $U(p)$ a fibration, then $U(f)$ is a weak equivalence. Then $\cD$ admits a model structure satisfying the following conditions:
\begin{enumerate}
\item[(F)]A map $p$ of $\cD$ is a fibration if and only if $U(p)$ is a fibration.
\item[(W)]A map $w$ of $\cD$ is a weak equivalence if and only if $U(w)$ is a weak equivalence.
\end{enumerate}
\end{thm}
A typical application is to transfer a symmetric monoidal model structure to the category of modules over a monoid of the given category.
\begin{thm}[\cite{SS} Theorem 4.1, \cite{L2} Proposition 4.1.4.3]\label{1.2.6}
Let $\cV$ be a combinatorial symmetric monoidal model category satisfying the monoid axiom, and $A$ be a monoid in $\cV$. Then the category of $A$-modules in $\cV$ inherits the structure of a combinatorial model structure satisfying the following conditions:
\begin{enumerate}
\item[(F)]A morphism of $A$-modules is a fibration if and only if it is a fibration as a morphism of $\cV$.
\item[(W)]A map of $A$-modules is a weak equivalence if and only if it is a weak equivalence as a morphism of $\cV$.
\end{enumerate}
\end{thm}
\begin{cor}\label{1.2.7}
Let $A$ be a dg ring. Then the category $A\cmod$ admits a proper combinatorial model structure satisfying the following conditions:
\begin{enumerate}
\item[(F)]A map is a fibration if and only if it is an epimorphism.
\item[(W)]A map is a weak equivalence if and only if it is a quasi-isomorphism.
\end{enumerate}
\end{cor}
For a relation of the two model structures, one can easily prove the following result:
\begin{prop}[\cite{Hov1}, \cite{A}]\label{1.2.8}
For any (dg) ring $A$ the injective and projective model structures on $A\cmod$ are Quillen equivalent.
\end{prop}
However, the two model structures are not equal in general. For instance every cofibration of $A\cmod$ with respect to the injective model structure is not necessarily a cofibration with respect to the projective model structure. It is known that the coincidence happens for example when $A$ is a field. For its proof we shall introduce a generalization of Theorem \ref{1.2.3} due to Mark Hovey:
\begin{thm}[\cite{CH} Theorem 2.2]\label{1.2.9}
Let $\cA$ be a bicomplete abelain category, and $\cP$ be a projective class on $\cA$. Then the category $C(\cA)$ together with the $\cP$-equivalences, the $\cP$-fibrations, and the $\cP$-cofibrations (in the sense of \cite{CH}) gives rise to a model category if and only if cofibrant replacements exist. Moreover, the resulting model structure is proper if it exists. The latter condition is satisfied in the following cases:
\begin{enumerate}
\item[(A)]$\cP$ is the pullback of the trivial projective class (i.e., the whole class of objects) along a right adjoint that preserves countable sums;
\item[(B)]There are enough $\kappa$-small $\cP$-projectives for some cardinal $\kappa$, and $\cP$-resolutions can be chosen functorially.
\end{enumerate}
\end{thm}
As a special case where $\cP$ is the trivial projective class we get the following consequence:
\begin{thm}[\cite{CH} Example 3.4]\label{1.2.10}
Let $\cA$ be a bicomplete abelian category, and $C(\cA)$ the category of complexes of objects of $\cA$. Then there exists a model structure on $C(\cA)$ satisfying the following conditions:
\begin{enumerate}
\item[(C)]A map is a cofibration if and only if it is a degreewise split monomorphism.
\item[(F)]A map is a fibration if and only if it is a degreewise split epimorphism.
\item[(W)]A map is a weak equivalence if and only if it is a homotopy equivalence.
\end{enumerate}
In particular, every object is both fibrant and cofibrant. Moreover, if $\cA$ is closed monoidal then this model structure is also monoidal for the standard monoidal structure on $C(\cA)$.
\end{thm}
We now set $\cA$ as the category of vector spaces over a field $k$. Then the resulting model structure on $k\cmod$ from Theorem \ref{1.2.10} is both injective and projective. Though Theorem \ref{1.2.9} does not assert its combinatoriality it follows in our case by combining Theorem \ref{1.2.1}.
\subsection{The Main Results}
In this paper we will achieve analogues of the above for our categories. In a view from the proof of Theorem \ref{1.2.2} and our aim to prove the combinatoriality, we will need the following easy fact at the first place:
\begin{lem}
Let $(\cA,K)$ be a (weak) pair. Then the category of (weak) $(\cA,K)$-modules is Grothendieck abelian. In particular, it is locally presentable by \cite{Be} Proposition 3.10.
\end{lem}
Then the next result is proved by the same way as Theorem \ref{1.2.2}.
\begin{thm}\label{1.3.2}
Let $(\cA,K)$ be a (weak) pair. Then the category of (weak) $(\cA,K)$-modules admits a proper combinatorial model structures satisfying the following conditions:
\begin{enumerate}
\item[(C)]A map is a cofibration if and only if it is a monomorphism.
\item[(W)]A map is a weak equivalence if and only if it is a quasi-isomorphism (i.e., a quasi-isomorphism as a homomorphism of cochain complexes of $k$-modules).
\end{enumerate}
We will call it the injective model structure.
\end{thm}
\begin{cor}\label{1.3.3}
The forgetful functors $\cF_{\cB,L,w}^{\cA,K}$ and $\cF_{\cB,L}^{\cA,K}$ of \cite{H1} are left Quillen with respect to the injective model structure. In particular, their right adjoints $I^{\cB,L}_{\cA,K,w}$ and $I^{\cB,L}_{\cA,K}$ admit the right derived functors.
\end{cor}
To put the projective model structure we work with a field $k$ of characteristic 0 and reductive algebraic groups $K$ over $k$.
\begin{thm}\label{1.3.4}
There is a proper combinatorial symmetric monoidal model structure on $K\cmod$ such that the following conditions are satisfied: 
\begin{enumerate}
\item[(C)]A map is a cofibration if and only if it is injective.
\item[(F)]A map is a fibration if and only if it is surjective.
\item[(W)]A map is a weak equivalence if and only if it is a quasi-isomorphism.
\end{enumerate}
\end{thm}
Using the transfer, we deduce the following result:
\begin{thm}\label{1.3.5}
If we are given a (weak) pair $(\cA,K)$ with $K$ reductive the category of (weak) $(\cA,K)$-modules admits a proper combinatorial model structure satisfying the following conditions:
\begin{enumerate}
\item[(F)]A map is a fibration if and only if it is surjective.
\item[(W)]A map is a weak equivalence if and only if it is a quasi-isomorphism.
\end{enumerate}
Moreover, it is Quillen equivalent with Theorem \ref{1.3.2}.
\end{thm}
\begin{cor}
The functors $P^{\cB,L}_{\cA,K,w}$ and $P^{\cB,L}_{\cA,K}$ are left Quillen. In particular, these admit the left derived functors.
\end{cor}
For a digression of Theorem \ref{1.3.4} we consider the symmetric monoidal full subcategory $K\cmod^{\leq 0}$ of $K\cmod$ spanned by complexes concentrated in nonpositive degrees. A motivation comes from derived (homotopical) algebraic geometry (\cite{TV}, \cite{L3}, \cite{GR}). In \cite{GR} for example, they use the commutative dg algebras concentrated in nonpositive degrees as a model of the coordinate rings of affine dg schemes. In this paper we acheive its equivariant analogue.
\begin{thm}
There is a right proper symmetric monoidal model structure on the category of commutative monoids of $K\cmod^{\leq 0}$ such that the following conditions are satisfied:
\begin{enumerate}
\item[(F)]A map is a fibration if and only if the underlying homomorphism of dg vector spaces is surjective.
\item[(W)]A map is a weak equivalence if and only if it is a quasi-isomorphism.
\end{enumerate}
\end{thm}

In \cite{H2} we will work on basic studies of the underlying $\infty$-category of Theorem \ref{1.3.2} and its dg enhancement.
\renewcommand{\abstractname}{Acknowledgements}
\begin{abstract}
The author is grateful to his advisor Hisayosi Matumoto for useful advice, and to Jun Yoshida for helpful discussions. He also would like to express his gratitude to Fabian Januszewski for encouraging him to work out this paper.
\end{abstract}
\section{Model structures}
We put model structures on
\[(\cA,K)\cmod_w,\ (\cA,K)\cmod,\]
whose weak equivalences are quasi-isomorphisms in order to define the unbounded derived functors of the functors $I^{\cB,L}_{\cA,K}$ and $P^{\cB,L}_{\cA,K}$ (see 1.1 for our conventions). We also give examples of homotopies. \cite{Hov1} is a standard and basic reference on model categories. The appendix of \cite{L1} is also useful.
\subsection{Local presentability}
This section is devoted to supplementary results to \cite{H1} towards constructions of the combinatorial model structures.

Let us fix a ground commutative ring $k$. Notice that for a weak pair $(\cA,K)$ over $k$ each Hom set $\Hom_{\cA,K}(M,N)$ of the category of weak $(\cA,K)$-modules is a $k$-submodule of the Hom $k$-module $\Hom_k(M,N)$ of $k\cmod$. With this $k$-module structure, the category of weak $(\cA,K)$-modules (and that of $(\cA,K)$-modules for a pair $(\cA,K)$) inherits a $k$-linear structure. Moreover, in the former paper \cite{H1} we showed that the category of weak $(\cA,K)$-modules (and that of $(\cA,K)$-modules for a pair $(\cA,K)$) is a locally small bicomplete abelian category, whose colimits and finite limits are computed in $k\cmod$ (\cite{H1} Corollary 2.3.9 (1)). In this section we verify the following stronger statements:
\begin{thm}\label{2.1.1}
\begin{enumerate}
\renewcommand{\labelenumi}{(\arabic{enumi})}
\item Let $(\cA,K)$ be a weak pair. Then the $k$-linear category of weak $(\cA,K)$-modules is a Grothendieck abelian category (\cite{Gro}).
\item Let $(\cA,K)$ be a pair. Then the $k$-linear category of $(\cA,K)$-modules is a Grothendieck abelian category.
\end{enumerate}
\end{thm}
\begin{proof}
In virtue of Lemma \ref{1.1.2} it will suffice to prove (1). Let $(\cA,K)$ be a weak pair. Then filtered colimits of the category of weak $(\cA,K)$-modules are exact since those of $k\cmod$ are exact. Existence of a generator of $(\cA,K)\cmod$ follows since $K\cmod$ is a Grothendieck abelian category. In fact, we have a unique monoid morphism from the unit monoid $(k,K)$ to $(\cA,K)$ (\cite{H1} Example 2.2.4 (2)), and the image of a generator of $K\cmod$ under $P^{\cA,K}_{k,K,w}$ is again a generator.
\end{proof}
\begin{cor}\label{2.1.2}
\begin{enumerate}
\renewcommand{\labelenumi}{(\arabic{enumi})}
\item For a weak pair $(\cA,K)$, the category $(\cA,K)\cmod$ is locally presentable.
\item For a pair $(\cA,K)$, the category $(\cA,K)\cmod$ is locally presentable.
\end{enumerate}
\end{cor}
\begin{proof}
The assertions follow from the fact that every Grothendieck abelian category is locally presentable (\cite{Be} Proposition 3.10).
\end{proof}
\begin{rem}\label{2.1.3}
We constructed right adjoint functors $I^{\cB,L}_{\cA,K,w}$ (resp.\ $I^{\cB,L}_{\cA,K}$) of $\cF^{\cA,K}_{\cB,L,w}$ (resp.\ $\cF^{\cA,K}_{\cB,L}$) by using the theory of monoidal categories in \cite{H1} Theorem 2.3.4. Instead, these right adjoint functors can be constructed directly by Corollary \ref{2.1.2} and the adjoint functor theorem.
\end{rem}
\subsection{Notations}
Here, we collect some additional notations related to the theory of model categories. Let $\cC$ be a category, and $S$ be a class of morphisms of $\cC$. Then we will say a map $p:X\to Y$ has the left lifting property with respect to $S$ if for any tuple of maps $i:A\to B$ of $S$, $f:A\to X$, and $g:B\to Y$ of $\cC$, which satisfies $p\circ f=g\circ i$, there exists a map $F:B\to X$ such that $F\circ i=f$ and $p\circ F=g$. 
\[\xymatrix{A\ar[r]^{\forall f}\ar[d]_i& X\ar[d]^p\\
B\ar@{-->}[ru]^{\exists F}\ar[r]_{\forall g} &Y}\]
In this paper, the class of maps satisfying the left lifting property with respect to $S$ will be denoted by $S_{\bot}$. Dually, we say $i:A\to B$ has the right lifting property with respect to $S$ if for any tuple of maps $p:X\to Y$ of $S$, $f:A\to X$ and $g:B\to Y$ of $\cC$, which satisfies $p\circ f=g\circ i$, there exists a map $F:B\to X$ such that $F\circ i=f$ and $p\circ F=g$. The class of maps satisfying the right lifting property with respect to $S$ will be denoted by ${}_\bot S$.
\subsection{The injective model structure}
\begin{thm}\label{2.3.1}
The category of (weak) $(\cA,K)$-modules admits a proper combinatorial model structure satisfying the following conditions:
\begin{enumerate}
\item[(C)]A map is a cofibration if and only if it is a monomorphism.
\item[(W)]A map is a weak equivalence if and only if it is a quasi-isomorphism.
\end{enumerate}
We will call it the injective model structure.
\end{thm}
This is proved in a similar way to \cite{L2} Proposition 1.3.5.3 and \cite{LS} Proposition 4.1. Let us recall some results of the appendix of \cite{L1}. We say that a collection $W$ of morphisms of a locally presentable category is perfect if the following conditions hold:
\begin{enumerate}
\renewcommand{\labelenumi}{(\roman{enumi})}
\item Every isomorphism of the given category belongs to $W$.
\item The collection $W$ has the two-out-of-three property.
\item The collection $W$ is stable under filtered colimits.
\item There exists a small subset $W_0\subset W$ such that $W$ is generated by $W_0$ via filtered colimits.
\end{enumerate}
\begin{ex}[\cite{L1} Example A.2.6.11]\label{2.3.2}
For a locally presentable category, the class of its isomorphisms is perfect.
\end{ex}
\begin{prop}[\cite{L1} Corollary A.2.6.12, \cite{AR} Remark 2.50, \cite{Be} Proposition 1.18]\label{2.3.3}
Let $F:\cC\to\cC'$ be a functor between locally presentable categories preserving filtered colimits, and $W_{\cC'}$ be a perfect class of morphisms in $\cC'$. Then $W_\cC=F^{-1}W_{\cC'}$ is perfect.
\end{prop}
\begin{thm}[\cite{L1} Corollary A.2.6.13, see also \cite{Be}]\label{2.3.4}
Let $\cC$ be a locally presentable category, $W$ be a perfect class of morphisms of $\cC$, and $I$ be a small set of morphisms of $\cC$. If these satisfy the following conditions, $\cC$ admits a left proper combinatorial model structure such that the class of cofibrations is the (weakly) saturated class generated by $I$ and that of weak equivalences is $W$.
\begin{enumerate}
\renewcommand{\labelenumi}{(\roman{enumi})}
\item Suppose that we are given a diagram
\[\xymatrix{X\ar[d]_f\ar[r]&X'\ar[d]\ar[r]^g&Y\ar[d]\\
Z\ar[r]&Z'\ar[r]_{g'}&V,}\] 
where the squares are pushout diagrams. Then if $f\in I$ and $g\in W$, $g'$ belongs to $W$.
\item If $f$ has the right lifting property with respect to $I$, $f$ belongs to $W$.
\end{enumerate}
\end{thm}
\begin{note}\label{2.3.5}
For a dg $k$-module $M$ and a homogeneous element $m\in M$ we will denote the degree of $m$ by $\overline{m}$.
\end{note}
\begin{cons}[shifts, \cite{P1} 5.2.1]\label{2.3.6}
Let $V$ be a weak $(\cA,K)$-module, and $n$ be an integer. Then we set a new weak $(\cA,K)$-module $V\left[n\right]$ by taking the shift $\left[n\right]$ both as a dg $\cA$-module and as a complex of $K$-modules. Namely, the complex $V\left[n\right]$ is defined as follows:
\begin{itemize}
\item $(V\left[n\right])^m=V^{m+n}$ for every integer $m$;
\item $d_{V\left[n\right]}(v)=(-1)^n d_V(v)$ for $v\in V$;
\item $\nu_{V\left[n\right]}(k)v=\nu_V(k)v$ for any $v\in V$ and $k\in K$;
\item $\pi_{V\left[n\right]}(a)v=(-1)^{\overline{a}n}\pi_V(a)v$ for any $v\in V$ and $a\in\cA$.
\end{itemize}
This is actually an $(\cA,K)$-module if so is $V$.
\end{cons}
\begin{cons}[mapping cones, \cite{P1} 5.2.1]\label{2.3.7}
We construct mapping cones of (weak) $(\cA,K)$-modules. Let $f:X\to Y$ be a homomorphism of weak $(\cA,K)$-modules. Then its mapping cone $\Cone(f)$ is defined as follows:
\begin{itemize}
\item $\Cone(f)=Y\oplus X\left[1\right]$ as a graded $K$-module;
\item $d_{\Cone(f)}(y,x)=(d_Y(y)+f(x),-d_X(x))$ for $(y,x)\in\Cone(f)$;
\item $\pi_{\Cone(f)}(a)(y,x)=(\pi_Y(a)y,(-1)^{\overline{a}}\pi_X(a)x)$ for any $x\in X$, $y\in Y$, and $a\in\cA$.
\end{itemize}
This is actually an $(\cA,K)$-module if so are $X$ and $Y$. If we are given a (weak) $(\cA,K)$-module $V$ we set $\Cyl(X)$ as the mapping cone of the map $V\to V\oplus V$ determined by the identity map $id_V$ and its minus $-id_V$. Explicitly, this is given as follows:
\begin{itemize}
\item $\Cyl(X)=X\oplus X\left[1\right]\oplus X$ as a graded $K$-module;
\item $d_{\Cyl(X)}(x,y,z)=(d_X(x)+y,-d_Y(x),d_X(z)-y)$;
\item $\pi_{\Cyl(X)}(a)(x,y,z)=(\pi_X(a)x,(-1)^{\overline{a}}\pi_X(a)y,\pi_X(a)z)$.
\end{itemize}
\end{cons}
\begin{rem}\label{2.3.8}
In \cite{H2} we will generalize these constructions above and characterize the resulting objects universally in the context of dg categories (\cite{BK}).
\end{rem}
\begin{lem}\label{2.3.9}
If a homomorphism $p:X\to Y$ of (weak) $(\cA,K)$-modules has the right lifting property with respect to every acyclic injection then $p$ is surjective.
\end{lem}
\begin{proof}
Let $p:X\to Y$ be a map of (weak) $(\cA,K)$-modules with the right lifting property with respect to every acyclic injection. Then the canonical projection
\[\Cone(id_Y)\left[-1\right]\to Y\]
lifts to a map $\Cone(id_Y)\left[-1\right]\to X$ since $\Cone(id_Y)$ is acyclic. The surjectivity now follows.
\[\xymatrix{&X\ar[d]^p\\
\Cone(id_Y)\left[-1\right]\ar[r]\ar@{-->}[ru]&Y
}\]
\end{proof}
\begin{proof}[proof of Theorem \ref{2.3.1}]
We check the conditions in Theorem \ref{2.3.4} to apply it. Here we consider the category $(\cA,K)\cmod$. The proof of the theorem for $(\cA,K)\cmod_w$ goes in the same way.

Firstly, recall that $(\cA,K)\cmod$ is locally presentable (Corollary \ref{2.1.2}), and notice that there exists a generator of the class of monomorphisms as a saturated class (Theorem \ref{2.1.1}, \cite{Be} Proposition 1.12 and Remark 1.13).

Secondly, we show that the class of quasi-isomorphisms of $(\cA,K)\cmod$ is perfect. Let $k\cmod^\heartsuit$ denote the category of $k$-modules, and
$H^n:(\cA,K)\cmod\to k\cmod^\heartsuit$ the $n$th cohomology functor for each integer $n$. Then the functor
\[\prod_n H^n:(\cA,K)\cmod\to\prod_nk\cmod^\heartsuit\]
commutes with filtered colimits since those of $k\cmod^\heartsuit$ are exact. If we set $W_{\prod k\cmod^\heartsuit}$ as the collection of isomorphisms, $W=(\prod_nH^n)^{-1}W_{\prod k\cmod^\heartsuit}$ is the collection of quasi-isomorphisms. Since $\prod k\cmod^\heartsuit$ is locally presentable, $W_{\prod k\cmod^\heartsuit}$ is perfect (Example \ref{2.3.2}). Therefore $W$ is also perfect (Proposition \ref{2.3.3}).

Thirdly, we check that a map $p:X\to Y$ with the right lifting property with respect to the generator of the class of monomorphisms is a quasi-isomorphism. From \cite{L1} Corollary A.1.2.7, it is equivalent to proving that a map with the right lifting property with respect to all injective homomorphisms is a quasi-isomorphism. Notice that $p$ is a surjective map. In fact, $p$ has a section since $p$ has the left lifting property with respect to the zero map $0\to Y$. Therefore it also suffices to prove that every injective object $I\in(\cA,K)\cmod$ is acyclic since $\Ker p$ is injective. Since $I$ is injective, the natural injective homomorphism $I\to \Cone(id_I)$ admits a retract $r:\Cone(id_I)\to I$. Since $\Cone(id_I)$ is acyclic, so is $I$.

To complete the proof it will suffice to show that pushouts (resp.\ pullbacks) of quasi-isomorphisms along injective (resp.\ surjective) homomorphisms are again quasi-isomorphisms (Lemma \ref{2.3.9}). These follow from \cite{Hov2} Corollary 1.4.
\end{proof}
The following lemma is obvious by definition. This says that we have the unbounded derived functors $\bR I^{\cB,L}_{\cA,K,w}$ and $\bR I^{\cB,L}_{\cA,K}$.
\begin{lem}\label{2.3.10}
The functors $\cF_{\cB,L,w}^{\cA,K}$, $\cF_{\cB,L}^{\cA,K}$ are left Quillen with respect to the injective model structures.
\end{lem}
As a close of this section, we give an example of a cylinder object, and compute a left homotopy with respect to it. For a dg algebra $\cA$ we denote its underlying graded algebra by $\cA^\sharp$. For (weak) $(\cA,K)$-modules $X,Y$ an $(\cA^\sharp,K)$-module homomorphism is a graded $K$-module homomorphism $X\to Y$ such that it is also a graded $\cA^\sharp$-module homomorphism. A chain homotopy of two homomorphisms $f,g:X\to Y$ of (weak) $(\cA,K)$-modules is a graded $(\cA^\sharp,K)$-homomorphism $h:X\to Y\left[-1\right]$ such that $h\circ d_M+d_N\circ h=f-g$. If there exists a chain homotopy of maps $f$ and $g$ then we will say that $f$ and $g$ are chain homotopic.
\begin{prop}\label{2.3.11}
Let $(X,\pi_X,\nu_X)$ be a (weak) $(\cA,K)$-module. Then $\Cyl(X)$ with the diagram
\[\xymatrix{X\oplus X\ar[rr]^+\ar[rd]&&X\\&\Cyl(X)\ar[ru]&}\]
is a cylinder object of $X$ in the sense of \cite{Hov1} Definition 1.2.4, where $X\oplus X\to\Cyl(X)$ is the natural embedding map, $\Cyl(X)\to X$ is the summation map of the two sides, and $\Cyl(X) \to X$ is given by summing up the two sided components of $\Cyl(X)$. Moreover, a left homotopy with respect to this cylinder object provides a chain homotopy of (weak) $(\cA,K)$-modules, and conversely a left homotopy with respect to this cylinder object is constructed from a chain homotopy. 
\end{prop}
\begin{proof}
The first half of the assertion is obvious. We show the other half. Here we consider weak $(\cA,K)$-modules with $(\cA,K)$ a weak pair for simplicity. Our arguments below also hold for $(\cA,K)$-modules.

Let $f,g:X\to Y$ be maps of weak $(\cA,K)$-modules. Notice that a graded $(\cA^\sharp,K)$-module homomorphism $H:\Cyl(X)\to Y$ such that the composition
\[X\oplus X\to\Cyl(X)\to Y\]
is $(f,g)$ is uniquely determined by a graded $(\cA^\sharp,K)$-module homomorphism $h:X\to Y\left[-1\right]$ and
\[H^n(x,y,z)=f^n(x)+h^{n+1}(y)+g^n(z)\]
for integers $n$. Furthermore, our $H:\Cyl(X)\to Y$ commutes with the differential if and only if the corresponding map $h$ is a chain homotopy from $f$ to $g$. In fact, we have
\[d^{n-1}_Y\circ H^{n-1}\circ\left(\begin{array}{ccc}
0\\
id\\
0\end{array}
\right)=d^{n-1}_Y\circ h^n\]
and
\[H^n\circ d^{n-1}_{\Cyl(X)}\circ\left(\begin{array}{ccc}
0\\
id\\
0\end{array}
\right)=-h^{n+1}\circ d^n_X+f^n-g^n\]
as maps from $X^n$ to $Y^n$ since $f$ and $g$ commutes with the differential.
\end{proof}
\subsection{The projective model structure I}
In the rest of this paper we assume that the base ring $k$ is a field of characteristic zero and that affine group schemes $K$ are reductive algebraic groups. In this section we study fibrations of our injective model structure on $K\cmod$.

We say that an abelian category $\cA$ is semisimple if every object of $\cA$ is injective. Thanks to Theorem \ref{1.2.10} we deduce the following model structure:
\begin{cor}\label{2.4.1}
Let $\cA$ be a bicomplete semisimple abelian category. Then there exists a model structure on $C(\cA)$ which is described as follows:
\begin{enumerate}
\item[(C)]A map is a cofibration if and only if it is a monomorphism.
\item[(F)]A map is a fibration if and only if it is an epimorphism.
\item[(W)]A map is a weak equivalence if and only if it is a quasi-isomorphism.
\end{enumerate}
\end{cor}
\begin{cor}\label{2.4.2}
Let $K$ be a reductive algebraic group over $k$. Then there is a proper combinatorial symmetric monoidal model structure on $K\cmod$ such that the following conditions are satisfied: 
\begin{enumerate}
\item[(C)]A map is a cofibration if and only if it is injective.
\item[(F)]A map is a fibration if and only if it is surjective.
\item[(W)]A map is a weak equivalence if and only if it is a quasi-isomorphism.
\end{enumerate}
In particular, every dg $K$-module is both fibrant and cofibrant. We sometimes call it also the projective model structure.
\end{cor}
\begin{cor}\label{2.4.3}
There is a proper combinatorial model structure on the category of monoids of $K\cmod$ such that the following conditions are satisfied:
\begin{enumerate}
\item[(F)]A map is a fibration if and only if the underlying homomorphism of dg vector spaces is surjective.
\item[(W)]A map is a weak equivalence if and only if it is a quasi-isomorphism.
\end{enumerate}
\end{cor}
\begin{proof}
Use \cite{SS} Theorem 4.1 (3) and \cite{L2} Proposition 4.1.4.3.
\end{proof}
As a variant we consider the category $K\cmod^{\leq 0}$ of dg $K$-modules concentrated in nonpositive degrees.
\begin{prop}\label{2.4.4}
There is a proper combinatorial symmetric monoidal model structure on $K\cmod^{\leq 0}$ such that the following conditions are satisfied:
\begin{enumerate}
\item[(C)]A map is a cofibration if and only if it is injective.
\item[(F)]A map is a fibration if and only if it is surjective in each negative degree.
\item[(W)]A map is a weak equivalence if and only if it is a quasi-isomorphism.
\end{enumerate}
In particular, every dg $K$-module is both fibrant and cofibrant.
\end{prop}
\begin{proof}
The model structure is obtained by applying \cite{Bo} 4.4. Since every dg $K$-module is both fibrant and cofibrant, this model structure is proper. The symmetric monoidalness follows from Corollary \ref{2.4.2}

To make this model structure combinatorial, observe that the collection of monomorphisms of $K\cmod^{\leq 0}$ is perfect by a similar argument to the proof of Theorem \ref{2.3.1}. The assertion now follows from \cite{L1} Corollary A.2.6.13.
\end{proof}
Finally, we transfer this model structure to the category $\CAlg(K\cmod^{\leq 0})$ of commutative monoids of $K\cmod^{\leq 0}$. For this, let us recall the notion of power cofibrations (\cite{L2} Definition 4.5.4.2). Given a pair of homomorphisms $f:A\to A'$ and $g:B\to B'$ of $K\cmod^{\leq 0}$, let $f\wedge g$ denote the induced map
\[(A\otimes B')\oplus_{A\otimes B}(A'\otimes B)\to A'\otimes B'.\]
For a cofibration $f:A\to B$ we iterate this operation to obtain a map
\[\wedge^n f:\Box^n(f)\to B^{\otimes n}.\]
Moreover, this is $\fS_n$-equivariant for the natural actions of $\fS_n$ on the target and the domain, where $\fS_n$ denotes the symmetric group of degree $n$. Let us regard $\fS_n$ as a groupoid, and put the projective model structure on the category $(K\cmod^{\leq 0})^{\fS_n}$ of diagrams of $K\cmod^{\leq 0}$ indexed by $\fS_n$. We will say that $f$ is a power cofibration if for every $n\geq 0$ the morphism $\wedge^n f$ is a cofibration of $(K\cmod^{\leq 0})^{\fS_n}$.
\begin{cor}\label{2.4.5}
There is a right proper symmetric monoidal model structure on the category of commutative monoids of $K\cmod^{\leq 0}$ such that the following conditions are satisfied:
\begin{enumerate}
\item[(F)]A map is a fibration if and only if the underlying homomorphism of dg vector spaces is surjective.
\item[(W)]A map is a weak equivalence if and only if it is a quasi-isomorphism.
\end{enumerate}
\end{cor}
\begin{proof}
In view of \cite{L2} Proposition 4.5.4.6 it will suffice to show that every cofibration is actually a power cofibration. We may identify the model category $(K\cmod^{\leq 0})^{\fS_n}$ with the category $K\times\fS_n\cmod^{\leq 0}$ equipped with the model structure of Corollary \ref{2.4.4} since $K\times\fS_n$ is reductive. The assertion is now obvious.
\end{proof}
\subsection{The projective model structure II}
From the previous section on, we assume that the ground ring $k$ is a field of characteristic $0$, and that all the group schemes $K$ in this section are reductive. We generalize Corollary \ref{2.4.2} to put the so-called projective model structure on $(\cA,K)\cmod_w$ and $(\cA,K)\cmod$. We recall some key ways to construct model structures.
\begin{thm}[\cite{Hov1} Theorem 2.1.19]\label{2.5.1}
Let $\cC$ be a locally presentable category, and $W$ be a collection of morphisms of $\cC$. We also let $I$ and $J$ be small sets of morphisms of $\cC$, whose domains and targets are small. Then there is a combinatorial model structure on $\cC$ with $I$ as a generator of the class of cofibrations, $J$ as a set of a generator of trivial cofibrations, and $W$ as the subcategory of weak equivalences if and only if the following conditions are satisfied:
\begin{enumerate}
\renewcommand{\labelenumi}{(\roman{enumi})}
\item The class $W$ has the two-out-of-three property and is closed under
retracts.
\item $({}_\bot J)_\bot\subset W\cap({}_\bot I)_\bot$.
\item ${}_\bot I\subset W\cap{}_\bot J$.
\item Either $W\cap({}_\bot I)_\bot\subset({}_\bot J)_\bot$ or $W\cap {}_\bot J\subset {}_\bot I$.
\end{enumerate}
\end{thm}
\begin{cor}[transfer of a model structure, \cite{GS} Theorem 3.6]\label{2.5.2}
Let $\cC$ be a combinatorial model category, $\cD$ a locally presentable category, and let $(F,U):\cC\to\cD$ be an adjunction. Suppose that if a map $f$ of $\cD$ has the left lifting property with respect to every map $p$ of $\cD$ with $U(p)$ a fibration, then $U(f)$ is a weak equivalence. Then $\cD$ admits a model structure satisfying the following conditions:
\begin{enumerate}
\item[(F)]A map $p$ of $\cD$ is a fibration if and only if $U(p)$ is a fibration.
\item[(W)]A map $w$ of $\cD$ is a weak equivalence if and only if $U(w)$ is a weak equivalence.
\end{enumerate}
\end{cor}
\begin{proof}
We set $I$ (resp. $J$) as a small set generating the class of cofibrations (resp.\ trivial cofibrations) of $\cC$. Let $F(I)$ (resp.\ $F(J)$) denote the set of the images of the maps of $I$ (resp.\ $J$) under $F$, and also let $W$ be the collection of maps of $\cD$ whose images under $U$ are weak equivalences of $\cC$. Then there is a combinatorial model structure on $\cD$ with $F(I)$ as a generator of the class of cofibrations, $F(J)$ as a generator of the class of trivial cofibrations and $W$ as the subcategory of weak equivalences. In fact, it suffices to check that they satisfy the conditions of Theorem \ref{2.5.1}.

The condition (i) is obvious since the collection of weak equivalences satisfies the two-out-of-three property and the stability under retracts. To prove the rest, observe that for a collection $S$ of morphisms of $\cC$, ${}_\bot F(S)$ precisely consists of the maps $p$ with $U(p)\in {}_\bot S$, where $F(S)$ is the image of $S$ under $F$. In particular, ${}_\bot F(I)$ (resp.\ ${}_\bot F(J)$) consists of maps $p$ in $\cD$ with $U(p)$ a trivial fibration (resp.\ a fibration). Part (ii) is just what we assumed. The conditions (iii) and (iv) are now obvious. This complete the proof.
\end{proof}
We are now ready to construct the projective model structure below.
\begin{thm}\label{2.5.3}
If we are given a (weak) pair $(\cA,K)$ with $K$ reductive the category of (weak) $(\cA,K)$-modules admits a proper combinatorial model structure satisfying the following conditions:
\begin{enumerate}
\item[(F)]A map is a fibration if and only if it is surjective.
\item[(W)]A map is a weak equivalence if and only if it is a quasi-isomorphism.
\end{enumerate}
\end{thm}
\begin{proof}
This will be obtained by transferring the projective model structure on $K\cmod$ via the standard adjunction $(P^{\cA,K}_{k,K,w}$, $\cF^{k,K}_{\cA,K,w}$) of Lemma \ref{1.1.3}. In view of Theorem \ref{2.3.1} and Lemma \ref{2.3.9} a map $f$ of $(\cA,K)\cmod_w$ satisfying the left lifting property with respect to all surjective maps is a quasi-isomorphism. In particular, the assumption of Corollary \ref{2.5.2} is satisfied, so that the desired model structure is obtained.
By a similar argument we can obtain the desired structure on $(\cA,K)\cmod$ with $(\cA,K)$ arbitrary by transfer of the projective model structure on $K\cmod$ by $\cF^{U(\fk),K}_{\cA,K}$ (\cite{H1} Example 2.2.4 (5)).

The properness follows by a similar argument to Theorem \ref{2.3.1}.
\end{proof}
As an application let us consider the functors $P^{\cB,L}_{\cA,K,w}$ and $P^{\cB,L}_{\cA,K}$. The next assertion immediately follows from Proposition \ref{1.1.4}:
\begin{lem}\label{2.5.4}
\begin{enumerate}
\renewcommand{\labelenumi}{(\arabic{enumi})}
\item For a weak map of weak pairs
\[(\cA,K)\to(\cB,L)\]
with $K$ and $L$ reductive, the associated functors $P^{\cB,L}_{\cA,K,w}$ is left Quillen with respect to the projective model structures.
\item For a map of pairs
\[(\cA,K)\to(\cB,L)\]
with $K$ and $L$ reductive, the associated functors $P^{\cB,L}_{\cA,K}$ is left Quillen with respect to the projective model structures.
\end{enumerate}
As a consequence we obtain the unbounded derived functors $\bL P^{\cB,L}_{\cA,K,w}$ and $\bL P^{\cB,L}_{\cA,K}$.
\end{lem}
Let $M$ be a weak $(\cA,K)$-module. Then set $\Path(M)$ as the $\left[-1\right]$-shift of the mapping cone of the homomorphism $M\oplus M\to M$ determined by $-id_M$ and $id_M$. With the diagonal map to the two sides $M\to\Path(M)$ and the projection map $\Path(M)\to M\times M$, $\Path(M)$ is a path object of $M$. Moreover, if $M$ is an $(\cA,K)$-module, so is $\Path(M)$. We are then able to compute right homotopies with respect to this path object in a similar way to Proposition \ref{2.3.11}.
\begin{prop}\label{2.5.5}
A right homotopy with respect to the path object constructed above provides a chain homotopy of (weak) $(\cA,K)$-modules. Conversely a right homotopy with respect to this path object is constructed from a chain homotopy.
\end{prop}
It is immediate from Lemma \ref{2.3.9} that the two kinds of model structures we constructed in this paper are equivalent:
\begin{thm}\label{2.5.6}
Suppose that $k$ is a field of characteristic $0$, and $(\cA,K)$ be a weak pair with $K$ reductive. Then the identity functor from $(\cA,K)\cmod_w$ with the injective model structure to $(\cA,K)\cmod_w$ with the projective model structure is a right Quillen equivalence. Similarly, for a pair $(\cA,K)$ with $K$ reductive, the identity functor from $(\cA,K)\cmod$ with the injective model structure to $(\cA,K)\cmod$ with the projective model structure is a Quillen equivalence.
\end{thm}

\end{document}